\documentclass{amsart}

\makeindex

\usepackage[all, cmtip]{xy}
\usepackage{amssymb}

\usepackage{algorithmic}
\usepackage{amsmath}

\usepackage{amsthm}

\usepackage{amscd}

\usepackage{amsfonts}

\usepackage{amssymb}

\usepackage[pdftex]{graphicx}

\usepackage{multirow}

\usepackage{color}

\usepackage{epic}

\usepackage{graphicx}

\usepackage{pgf,tikz}
\usetikzlibrary{arrows}

\newtheorem{theorem}{Theorem}[section]
\newtheorem{lemma}[theorem]{Lemma}
\newtheorem{proposition}[theorem]{Proposition}

\newtheorem{question}[theorem]{Question}
\theoremstyle{corollary}
\newtheorem{corollary}[theorem]{Corollary}

\theoremstyle{definition}     
\newtheorem{definition}[theorem]{Definition}

\newtheorem{example}[theorem]{Example}

\theoremstyle{remark}
\newtheorem{remark}[theorem]{Remark}

\numberwithin{equation}{section}

\newcommand{\C}{\mathbb{C}}

\newcommand{\N}{\mathbb{N}}
\newcommand{\Q}{\mathbb{Q}}
\def\P{\mathbb{P}}

\def\Proj{\operatorname{Proj}}

\def\mult{\operatorname{mult}}

\def\cha{\operatorname{char}}

\input xy
\xyoption{all}

\title[Redundant blow-ups of big anticanonical rational surfaces]{Redundant blow-ups of rational surfaces with big anticanonical divisor}

\begin{document}

\author{DongSeon Hwang}
\address{Department of Mathematics, Ajou University, Suwon, Korea}
\email{dshwang@ajou.ac.kr}

\author{Jinhyung Park}
\address{Department of Mathematical Sciences, KAIST, Daejeon, Korea}
\email{parkjh13@kaist.ac.kr}


\date{\today}


\begin{abstract}
We completely classify redundant blow-ups appearing in the theory of rational surfaces with big anticanonical divisor due to Sakai. In particular, we construct a rational surface with big anticanonical divisor which is not a minimal resolution of a del Pezzo surface with only rational singularities, which gives a negative answer to a question raised in a paper by Testa, V\'{a}rilly-Alvarado, and Velasco.
\end{abstract}

\maketitle

\tableofcontents

\section{Introduction}
Throughout the paper, we work over an algebraically closed field $k$ of arbitrary characteristic. Sakai (\cite[Proposition 4.1 and Theorem 4.3]{Sak84}) proved that the anticanonical morphism $f \colon S \to \bar{S}$ of a big anticanonical rational surface, i.e., a smooth projective rational surface with big anticanonical divisor, factors through the minimal resolution of a del Pezzo surface with only rational singularities followed by a sequence of \emph{redundant blow-ups}, which will be defined below. However, the existence of redundant points was not known before. In the present paper, we show the existence of redundant points (Theorem \ref{reddisc}) by providing a systematic way of finding redundant points (Theorem \ref{redpt}).

Here, we briefly introduce the notion of redundant blow-up in general inspired by Sakai's work. Let $S$ be a smooth projective surface. Assume that $-K_S$ is pseudo-effective so that we have the Zariski decomposition $-K_S = P+N$. A point $p$ in $S$ is called a \emph{redundant point} if $\mult_p N \geq 1$. The blow-up $f \colon \widetilde{S} \to S$ at a redundant point $p$ is called a \emph{redundant blow-up}. For more detail, see Section \ref{setupsec}.

To classify redundant blow-ups, we only need to know the information of the redundant points on the surface $S$. In many natural situations, $S$ is a minimal resolution of a normal projective surface $\bar{S}$ with nef anticanonical divisor. It turns out that the position of redundant points on $S$ can be read off from the information of singularities on $\bar{S}$.

\begin{theorem}\label{redpt}
Let $\bar{S}$ be a normal projective rational surface with nef anticanonical divisor, and let $\widetilde{S}$ be a surface obtained by a sequence of redundant blow-ups from the minimal resolution $S$ of $\bar{S}$. Then, we have the following.
 \begin{enumerate}
  \item The number of surfaces obtained by a sequence of redundant blow-ups from $S$ is finite if and only if $\bar{S}$ contains at worst log terminal singularities.
  \item If $\bar{S}$ contains at worst log terminal singularities, then every redundant point on $\widetilde{S}$ lies on the intersection points of the two curves contracted by the morphism $h \colon \widetilde{S} \rightarrow \bar{S}$.
  \item If $\bar{S}$ contains a non-log terminal singularity, then there is a curve in $\widetilde{S}$ contracted by the morphism $h \colon \widetilde{S} \rightarrow \bar{S}$ such that every point lying on this curve is a redundant point.
 \end{enumerate}
\end{theorem}

Moreover, the existence of the redundant points can also be determined from the singularity types of $\bar{S}$.

\begin{theorem}\label{reddisc}
Let $\bar{S}$ be a normal projective rational surface with nef anticanonical divisor, and let $g \colon S \rightarrow \bar{S}$ be its minimal resolution. Then, $S$ has no redundant point if and only if $\bar{S}$ contains at worst canonical singularities or log terminal singularities whose dual graphs are as follows:

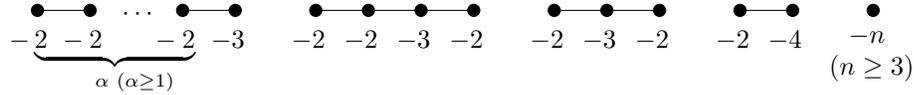
\begin{tikzpicture}[line cap=round,line join=round,>=triangle 45,x=1.0cm,y=1.0cm]
\clip(-4,2.8) rectangle (9,4.2);
\draw (-3.5,3.96)-- (-2.8,3.96);
\draw (-1.57,3.96)-- (-0.87,3.96);
\draw (-4,3.82) node[anchor=north west] {$-\underbrace{2\text{  }-2 \text{   }\text{   }\text{   }\text{   }\text{   } -2}_{\text{ } \text{ } \text{ } \text{ } \text{ } \alpha \text{ } (\alpha \geq 1)}$};
\draw (-1.32, 3.82) node[anchor=north west] {$-3$};
\draw (-2.5,4.12) node[anchor=north west] {$\cdots$};
\draw (0.2,3.96)-- (2.3,3.96);
\draw (-0.25,3.82) node[anchor=north west] {$-2$};
\draw (0.45,3.82) node[anchor=north west] {$-2$};
\draw (1.15,3.82) node[anchor=north west] {$-3$};
\draw (1.85,3.82) node[anchor=north west] {$-2$};
\draw (3.37,3.96)-- (4.77,3.96);
\draw (2.92,3.82) node[anchor=north west] {$-2$};
\draw (3.62,3.82) node[anchor=north west] {$-3$};
\draw (4.32,3.82) node[anchor=north west] {$-2$};
\draw (5.84,3.96)-- (6.54,3.96);
\draw (5.39,3.82) node[anchor=north west] {$-2$};
\draw (6.09,3.82) node[anchor=north west] {$-4$};
\draw (7.16,3.86) node[anchor=north west] {$-n$};
\draw (6.9,3.5) node[anchor=north west] {$(n \geq 3)$};
\begin{scriptsize}
\fill [color=black] (-3.5,3.96) circle (2.5pt);
\fill [color=black] (-2.8,3.96) circle (2.5pt);
\fill [color=black] (-1.57,3.96) circle (2.5pt);
\fill [color=black] (-0.87,3.96) circle (2.5pt);
\fill [color=black] (0.2,3.96) circle (2.5pt);
\fill [color=black] (0.9,3.96) circle (2.5pt);
\fill [color=black] (1.6,3.96) circle (2.5pt);
\fill [color=black] (2.3,3.96) circle (2.5pt);
\fill [color=black] (3.37,3.96) circle (2.5pt);
\fill [color=black] (4.07,3.96) circle (2.5pt);
\fill [color=black] (4.77,3.96) circle (2.5pt);
\fill [color=black] (5.84,3.96) circle (2.5pt);
\fill [color=black] (6.54,3.96) circle (2.5pt);
\fill [color=black] (7.61,3.96) circle (2.5pt);
\end{scriptsize}
\end{tikzpicture}
\end{theorem}

It was shown that every big anticanonical rational surface has a finitely generated Cox ring in \cite[Theorem 1]{TVV10}, \cite[Theorem 3]{CS08}, and the following question was raised in \cite{TVV10}.

\begin{question}[{\cite[Remark 3]{TVV10}}]\label{tvvq}
Is every big anticanonical rational surface the minimal resolution of a del Pezzo surface with only rational singularities?
\end{question}

We give a negative answer to this question by explicitly constructing examples of redundant blow-ups (see Subsection \ref{exredsubsec}).

\begin{theorem}\label{nomin}
For each $n \geq 10$, there exists a big anticanonical rational surface of Picard number $n$ which is not a minimal resolution of a del Pezzo surface with only rational singularities.
\end{theorem}

The remainder of this paper is organized as follows. We first define redundant blow-ups in Section \ref{setupsec}. Then, Section \ref{redptsec} is devoted to the investigation of redundant points with respect to a given singularity, which leads us to the proof of Theorem \ref{redpt}. In Section \ref{discsec}, we prove Theorem \ref{reddisc} by calculating discrepancies. Finally, in Section \ref{exredsec}, we construct examples of redundant blow-ups. In particular, we prove Theorem \ref{nomin} in Subsection \ref{exredsubsec}.

\subsection*{Acknowledgements}
The authors would like to thank Damiano Testa and Ivan Cheltsov for useful comments.

\section{Basic setup}\label{setupsec}

In this section, we define redundant blow-ups. Let $S$ be a smooth projective surface, and let $D$ be a $\Q$-divisor. The \emph{Iitaka dimension} of $D$ is given by
$$\kappa(D) := \max \{\dim \varphi_{|-nD| (S)} : n \in \N\},$$
whose value is one of $2,1,0,$ and $-\infty$.

We will frequently use the notion of the \emph{Zariski decomposition} of a pseudo-effective $\Q$-divisor $D$ (see \cite[Section 2]{Sak84} for details): $D$ can be written uniquely as $P+N$, where $P$ is a nef $\Q$-divisor, $N$ is an effective $\Q$-divisor, $P.N=0$, and the intersection matrix of the irreducible components of $N$ is negative definite if $N \neq 0$.

Write an effective $\Q$-divisor $D=\sum_{i=1}^{n} \alpha_i E_i$, where $E_i$ is a prime divisor for all $1 \leq i \leq n$. The \emph{multiplicity} of $D$ at a point $p$ in $S$ is defined by $\mult_{p} D := \sum_{i=1}^n \alpha_i \mult_{p} E_i$, where $\mult_{p} E_i$ denotes the usual multiplicity of $E_i$ at $p$.

In the remainder of this subsection, we use the following notations. Let $S$ be a smooth projective rational surface with $\kappa(-K_S) \geq 0$, and let $-K_S = P + N$ be the Zariski decomposition. Let $f \colon \widetilde{S} \to S$ be a blow-up at a point $p$ in $S$ with the exceptional divisor $E$.

\begin{definition}\label{reddef}
A point $p$ is called \emph{redundant} if $\mult_p N \geq 1$. The blow-up $f \colon \widetilde{S} \rightarrow S$ at a redundant point $p$ is called a \emph{redundant blow-up}, and the exceptional curve $E$ is called a \emph{redundant curve}.
\end{definition}

Note that we always have $\kappa(-K_{\widetilde{S}})  \leq \kappa(-K_S)$ in general. If $f$ is a redundant blow-up, then $\kappa(-K_{\widetilde{S}}) \geq 0$ by \cite[Lemma 6.9]{Sak84}.

Now, we reformulate Sakai's result on redundant blow-ups. This will play a key role throughout the paper.

\begin{lemma}[{\cite[Corollary 6.7]{Sak84}}]\label{redlem}
Assume that $\kappa(-K_{\widetilde{S}}) \geq 0$ so that we have the Zariski decomposition  $-K_{\widetilde{S}}=\widetilde{P}+\widetilde{N}$. Then, the following are equivalent:
\begin{enumerate}
 \item $f$ is a redundant blow-up.
 \item $\widetilde{P}= f^{*}P$ and $\widetilde{N}= f^{*}N - E$.
\end{enumerate}
\end{lemma}

\begin{remark}\label{redrem}
If $-K_{\widetilde{S}}$ is big, then $\widetilde{P}.E=0$ if and only if $f$ is a redundant blow-up by Lemma \ref{redlem}. Thus, our definition of redundant curves coincides with that of Sakai (\cite[Definition 4.1]{Sak84}).
\end{remark}

\section{Finding redundant points}\label{redptsec}

In this section, we focus on redundant points, and we prove Theorem \ref{redpt} at the end. In what follows, $\bar{S}$ denotes a normal projective rational surface such that $-K_{\bar{S}}$ is nef, and $g \colon S \rightarrow \bar{S}$ denotes the minimal resolution.

First, we recall some basic facts concerning normal singularities on surfaces. Let $(\bar{S}, s)$ be a germ of a normal surface singularity, and let $g \colon S \rightarrow \bar{S}$ be the minimal resolution. Denote the exceptional set by $A=\pi^{-1}(s)=E_1 \cup \cdots \cup E_l$, where each $E_i$ is an irreducible component. Note that $E_i^2 = -n_i \leq -1$, where each $n_i$ is an integer for $i=1,\ldots,l$.
Now, we have
$$-K_S = g^{*}(-K_{\bar{S}}) + \sum_{i=1}^{l} a_i E_i ,$$
where each $a_i$ is a nonnegative rational number for all $i=1,\ldots,l$. Here, we call $a_i$ the \emph{discrepancy} of $E_i$ with respect to $\bar{S}$ for convenience in computation, even though $-a_i$ is called the discrepancy in the literature. By the adjunction formula, this number can be obtained by the simultaneous linear equations:
$$\sum_{j=1}^{l}a_j E_j E_i = -K_{S}.E_i = -n_i + 2 \hbox{ for } i=1,\ldots,l.$$
Thus, each discrepancy $a_i$ can be calculated by the following matrix equation
\begin{displaymath}
\left( \begin{array}{cccc}
-n_1 & E_2 E_1 & \cdots & E_l E_1 \\
E_2 E_1 &-n_2 & \cdots & E_l E_2 \\
\vdots & \vdots & \ddots & \vdots \\
E_l E_1 & E_l E_2 & \cdots & -n_l
\end{array} \right)
\left( \begin{array}{c} a_1 \\ a_2 \\\vdots \\ a_l \end{array} \right) =
-\left( \begin{array}{c} n_1 - 2  \\ n_2 - 2  \\\vdots \\ n_l - 2  \end{array} \right).
\end{displaymath}

If an algebraic surface contains at worst rational singularities, then its singularities are isolated and the surface is $\Q$-factorial (\cite[Theorem 4.6]{B01}), and the exceptional set $A$ consists of smooth rational curves with a simple normal crossing (\emph{snc} for short) support. Throughout this paper, we adopt the following standard terminologies.

\begin{definition}\label{sing}
(1) A singularity $s$ is called \emph{canonical} if $a_i = 0 $ for all $i=1,\ldots,l$.\\
(2) A singularity $s$ is called \emph{log terminal} if $0 \leq a_i < 1$ for all $i=1,\ldots,l$.
\end{definition}

We note that every log terminal singularity is rational (\cite[Theorem 4.12]{KM98}). When $\cha(k)=0$, a log terminal singularity is nothing but a quotient singularity (\cite[Proposition 4.18]{KM98}).

We can get the Zariski decomposition of $-K_S$ immediately by the assumption that $-K_{\bar{S}}$ is nef.

\begin{lemma}\label{zar}
Let $P=g^{*}(-K_{\bar{S}})$ and $N=\sum_{i=1}^{l} a_i E_i$, where each $E_i$ is a $g$-exceptional curve and each $a_i$
is the discrepancy of $E_i$ with respect to $\bar{S}$  for $1 \leq i \leq l$. Then, $-K_S = P+N$ is the Zariski decomposition.
\end{lemma}

\begin{proof}
The pull-back of a nef divisor is again nef. Clearly $P.N=0$ and the intersection matrix of irreducible components of $N$ is negative definite.
\end{proof}

Suppose that $S$ contains a redundant point $p$. Let $f \colon  \widetilde{S} \rightarrow S$ be the redundant blow-up at $p$ with the exceptional divisor $E$, and let $-K_S = P+N$ and $-K_{\widetilde{S}} = \widetilde{P} + \widetilde{N}$ be the Zariski decompositions. Then, we obtain
$$\widetilde{N} = \sum_{i=1}^{l} a_i \widetilde{E}_i + (\mult_p N-1) E= \sum_{i=1}^{l} a_i \widetilde{E}_i + \left( \sum_{p \in E_j} a_j  -1\right) E,$$
where $\widetilde{E}_i$ is the strict transform of $E_i$ for each $1 \leq i \leq l$. We note that $\sum_{p \in E_j} a_j  \geq 1$,  since $\mult_p N \geq 1$.

\begin{remark}
If $N$ has an snc support, then so does $\widetilde{N}$. Thus, the negative part of the Zariski decomposition of the anticanonical divisor of a big anticanonical rational surface also has an snc support. This is no longer true for rational surfaces with anticanonical Iitaka dimension $1$ or $0$ (see Example \ref{0redex}).
\end{remark}

In the remainder of this section, we will completely determine the location of redundant points on $S$ in terms of the singularities of $\bar{S}$.

\subsection{$\bar{S}$ contains at worst canonical singularities}

\begin{lemma}\label{can}
If $\bar{S}$ has at worst canonical singularities, then $S$ has no redundant points.
\end{lemma}

\begin{proof}
By Definition \ref{sing}, $N = 0$; thus, the lemma follows.
\end{proof}

\subsection{$\bar{S}$ contains at worst log terminal singularities}

\begin{lemma}\label{lt}
If $\bar{S}$ contains at worst log terminal singularities, then any surface $\widetilde{S}'$ obtained by a sequence of redundant blow-ups from $S$ has finitely many redundant points, and every redundant point is an intersection point of two curves contracted by the morphism $\widetilde{S}' \rightarrow \bar{S}$.
\end{lemma}

\begin{proof}
By Definition \ref{sing}, $0 \leq a_i < 1$ for each $1 \leq i \leq l$. Finding a point $p$ in $S$ satisfying $\mult_p N \geq 1$ is equivalent to finding an intersection point of $E_j$ and $E_k$ such that $a_j + a_k \geq 1$ for some $1 \leq j \leq l$ and $1 \leq k \leq l$. Since the number of $E_i$'s is finite, the number of redundant points is also finite. In this case, we have
$$\widetilde{N}=\sum_{i=1}^{l}a_i \widetilde{E}_i + (a_j + a_k -1)E.$$
Observe that $a_i < 1$ for each $1 \leq i \leq l$ and $a_j+a_k-1 < 1$. Thus, after performing redundant blow-ups, the number of redundant points is still finite.
\end{proof}

\begin{lemma}\label{length}
The number of surfaces obtained by a sequence of redundant blow-ups from $S$ is finite.
\end{lemma}

\begin{proof}
Let $\widetilde{p}:= \widetilde{E}_j \cap E$ the intersection point. Then, we have
$$\mult_{\widetilde{p}} \widetilde{N}=a_j + (\mult_p N-1)=\mult_p N-(1-a_j)<\mult_p  N,$$
i.e., the multiplicity strictly decreases after redundant blow-ups.

Suppose that $\mult_{\widetilde{p}}\widetilde{N} \geq 1$. Let $f' \colon \widetilde{S}' \rightarrow \widetilde{S}$ be a redundant blow-up at $\widetilde{p}$ with the redundant curve $F$, and let $\widetilde{E}_j', \widetilde{E}_k'$ and $E'$ be the strict transforms of $\widetilde{E}_j, \widetilde{E}_k$ and $E$, respectively. Let $\widetilde{p}':=\widetilde{E}_j' \cap F$ and $q:=E' \cap F$ be the intersection points.

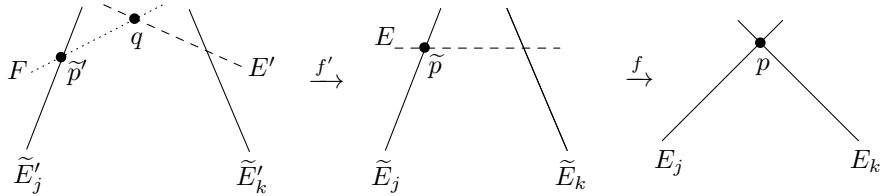
\begin{tikzpicture}[line cap=round,line join=round,>=triangle 45,x=0.8cm,y=0.8cm]
\clip(-7.4,0.9) rectangle (8,4.7);
\draw (5.74,4)-- (3.74,2);
\draw (5,4)-- (7,2);
\draw (0.06,4.2)-- (-0.86,1.84);
\draw (1.16,4.2)-- (2.16,1.84);
\draw (1.16,4.2)-- (2.16,1.84);
\draw [dash pattern=on 3pt off 3pt]  (-0.7,3.54)-- (2.04,3.54);
\draw (2.96,3.6) node[anchor=north west] {$\xrightarrow{f}$};
\draw (3.45,2.00) node[anchor=north west] {$E_j$};
\draw (6.7,2.00) node[anchor=north west] {$E_k$};
\draw (5.15,3.52) node[anchor=north west] {$p$};
\draw (-0.3,3.6) node[anchor=north west] {$\widetilde{p}$};
\draw (-1.25,1.84) node[anchor=north west] {$\widetilde{E}_j$};
\draw (1.8,1.84) node[anchor=north west] {$\widetilde{E}_k$};
\draw (-5.9,4.22)-- (-6.82,1.86);
\draw (-4.12,4.18)-- (-3.12,1.82);
\draw (-2.3,3.6) node[anchor=north west] {$\xrightarrow{f'}$};
\draw [dotted] (-6.74,3.14)-- (-4.52,4.28);
\draw [dash pattern=on 3pt off 3pt]  (-5.54,4.26)-- (-3.22,3.22);
\draw (-7.2,1.88) node[anchor=north west] {$\widetilde{E}_j'$};
\draw (-3.5,1.84) node[anchor=north west] {$\widetilde{E}_k'$};
\draw (-6.3,3.5) node[anchor=north west] {$\widetilde{p}'$};
\draw (-5.25,3.96) node[anchor=north west] {$q$};
\draw (-1.2,4.05) node[anchor=north west] {$E$};
\draw (-3.3,3.56) node[anchor=north west] {$E'$};
\draw (-7.3,3.48) node[anchor=north west] {$F$};
\begin{scriptsize}
\fill (5.37,3.63) circle (2.0pt);
\fill  (-0.21,3.55) circle (2.0pt);
\fill  (-5.03,4.03) circle (2.0pt);
\fill  (-6.25,3.39) circle (2.0pt);
\end{scriptsize}
\end{tikzpicture}
\\
Let $-K_{\widetilde{S}'} = \widetilde{P}' + \widetilde{N}'$ be the Zariski decomposition. Then, we have
\begin{displaymath}
\begin{array}{lll}
\mult_{\widetilde{p}} \widetilde{N}-\mult_{\widetilde{p}'} \widetilde{N}'&=&\mult_{\widetilde{p}}\widetilde{N}-\{\mult_{\widetilde{p}}\widetilde{N}-(1-a_j)\} = 1-a_j,\\
\mult_{\widetilde{p}}\widetilde{N}-\mult_{q}\widetilde{N}'&=&\mult_{\widetilde{p}}\widetilde{N}-\{(\mult_{\widetilde{p}}\widetilde{N}-1) + (\mult_pN -1)\}\\
&=&2-\mult_pN\\
&=& 2-a_j - a_k.
\end{array}
\end{displaymath}
Note that $\mult_p N-\mult_{\widetilde{p}}\widetilde{N}=1-a_j < 2-a_j-a_k,$
and hence, we obtain
$$\mult_p N-\mult_{\widetilde{p}}\widetilde{N} \leq \mult_{\widetilde{p}}\widetilde{N}-\mult_{\widetilde{p}'}\widetilde{N}',$$
and
$$\mult_pN-\mult_{\widetilde{p}}\widetilde{N} \leq \mult_{\widetilde{p}}\widetilde{N}-\mult_{q}\widetilde{N}', $$
i.e., the difference of multiplicities increases after redundant blow-ups. Thus, the assertion immediately follows.
\end{proof}

There exist natural numbers $M_j$ and $M_k$ such that
\[ \begin{array}{l}
\mult_{p}N- M_j (1-a_j) < 1 \hbox{ and} \mult_{p}N-(M_j -1) (1-a_j) \geq 1,
\end{array}\]
and
\[ \begin{array}{l}
\mult_{p}N- M_k (1-a_k) < 1  \hbox{ and} \mult_{p}N-(M_k -1) (1-a_k) \geq 1.
\end{array}\]
Since $\max \{M_j, M_k \}$ depends only on $p$,  we denote it by $M(p)$.

\begin{corollary}\label{bdm}
The maximal length of sequences of redundant blow-ups from $S$ is equal to
$$\max_{p \in R} \{M(p)\},$$
where $R$ is the set of all redundant points on $S$.
\end{corollary}

\begin{remark}
Corollary \ref{bdm} shows that there is a bound of the length of sequences of redundant blow-ups for a given surface $S$. However, there is no global bound for $M(p)$ (see Example \ref{ltex}; we have $M_1 > m-2 - \frac{2m}{n}$ and $M_2 > n-2 - \frac{2n}{m}$, and hence, $M(p)$ can be increased arbitrarily large as $m$ and $n$ goes to infinity).
\end{remark}

\subsection{$\bar{S}$ contains worse than log terminal singularities}

\begin{lemma}\label{rtsing}
If $\bar{S}$ contains a singularity that is not a log terminal singularity, then any surface $\widetilde{S}'$ obtained by a sequence of redundant blow-ups from $S$ contains a curve $C$ contracted by the morphism $\widetilde{S}' \rightarrow \bar{S}$ such that every point in $C$ is a redundant point. In particular, $\widetilde{S}'$ has infinitely many redundant points.
\end{lemma}

\begin{proof}
By Definition \ref{sing}, we can choose an integer $k$ with  $1 \leq k \leq l$ such that $a_k \geq 1$, and hence, every point in $E_k$ is a redundant point. After a redundant blow-up $\widetilde{S} \rightarrow S$ at a point $p$ in $E_k$, every point in the proper transform of $E_k$ is again a redundant point in $\widetilde{S}$. In this way, the lemma easily follows.
\end{proof}

We are ready to prove Theorem \ref{redpt}.

\begin{proof}[Proof of Theorem \ref{redpt}]
By Lemmas \ref{lt}, \ref{length}, \ref{rtsing}, the theorem holds.
\end{proof}

\section{Existence of redundant points}\label{discsec}

In this section, we focus on the existence of redundant points, and we prove Theorem \ref{reddisc} at the end. As in Section \ref{redptsec}, we use the following notations throughout this section: $\bar{S}$ denotes a normal projective rational surface such that $-K_{\bar{S}}$ is nef, and $g \colon S \rightarrow \bar{S}$ denotes the minimal resolution.

To prove Theorem \ref{reddisc}, by Lemmas \ref{can} and \ref{rtsing}, we only need to consider the case when $\bar{S}$ has at worst log terminal singularities. Recall that in this case, $p \in S$ is a redundant point if and only if there are two intersecting irreducible exceptional curves $E_j$ and $E_k$ in the minimal resolution such that the sum of discrepancies $a_j + a_k \geq 1$. Thus, it suffices to consider the problem locally near each singular point $s$ in $\bar{S}$, and hence, we can throughoutly assume that $\bar{S}$ contains only one log terminal singular point $s$.


In the case of characteristic zero, Brieskorn completely classified finite subgroups of $GL(2, k)$ without quasi-reflections, i.e., he classified all the dual graphs of quotient singularities of surfaces (\cite{Bri68}). It turns out that the complete list of the dual graphs of the log terminal surface singularities remains the same in arbitrary characteristic (see \cite{Ale91}). For the reader's convenience, we will give a detailed description of all possible types ($A_{q,q_1}, D_{q,q_1}, T_m, O_m$ and $I_m$) of dual graphs and discrepancies of log terminal singularities in the following subsections.

\subsection{$A_{q,q_1}$-type}\label{ltAsss}

The dual graph of a log terminal singularity $s \in \bar{S}$ of type $A_{q,q_1}$ is

\begin{tikzpicture}[line cap=round,line join=round,>=triangle 45,x=1.0cm,y=1.0cm]
\clip(-4.3,2.2) rectangle (7.28,4.0);
\draw (-1,3)-- (0,3);
\draw (0,3)-- (1,3);
\draw (2.58,3)-- (3.58,3);
\draw (-1.24,3.66) node[anchor=north west] {$E_1$};
\draw (-0.24,3.66) node[anchor=north west] {$E_2$};
\draw (0.78,3.66) node[anchor=north west] {$E_3$};
\draw (1.45,3.16) node[anchor=north west] {$\cdots$};
\draw (2.28,3.66) node[anchor=north west] {$E_{l-1}$};
\draw (3.36,3.66) node[anchor=north west] {$E_l$};
\draw (-1.3,2.84) node[anchor=north west] {$-n_1$};
\draw (-0.3,2.84) node[anchor=north west] {$-n_2$};
\draw (0.68,2.84) node[anchor=north west] {$-n_3$};
\draw (2.1,2.84) node[anchor=north west] {$-n_{l-1}$};
\draw (3.36,2.84) node[anchor=north west] {$-n_l$};
\begin{scriptsize}
\fill [color=black] (-1,3) circle (2.5pt);
\fill [color=black] (0,3) circle (2.5pt);
\fill [color=black] (1,3) circle (2.5pt);
\fill [color=black] (2.58,3) circle (2.5pt);
\fill [color=black] (3.58,3) circle (2.5pt);
\end{scriptsize}
\end{tikzpicture}\\
where each $n_i \geq 2$ is an integer for all i. This singularity can be characterized by the so-called \emph{Hirzebruch-Jung continued fraction}
 \[
 \frac{q}{q_1}=[n_1, n_2, ..., n_l]= n_1 - \dfrac{1}{n_2-\dfrac{1}{\ddots -
\dfrac{1}{n_l}}}.
 \]
The intersection matrix is
\begin{displaymath}
M(-n_1, \ldots, -n_l) := \left( \begin{array}{cccccc}
-n_1 & 1 & 0 & \cdots & \cdots &0 \\
1 & -n_2 & 1 & \cdots & \cdots& 0 \\
0 & 1 & -n_3 & \cdots & \cdots& 0 \\
\vdots & \vdots & \vdots & \ddots & \vdots & \vdots\\
0 & 0& 0 & \cdots  & -n_{l-1}&1  \\
0 & 0& 0 & \cdots &1 & -n_l \\
\end{array} \right).
\end{displaymath}
For simplicity, we use the notation $[n_1, \ldots, n_l]$ to refer to a log terminal singularity of $A_{q,q_1}$-type. Note that $[n_1, \ldots, n_l]$ and $[n_l, \ldots, n_1]$ denote the same singularity.

We will use the following notation for convenience (cf. \cite{HK09}).
\begin{enumerate}
\item $q_{b_1, b_2, \ldots, b_m} := |\det(M')|$ where $M'$ is the $(l-m)\times(l-m)$ matrix obtained by deleting $-n_{b_1}, -n_{b_2}, \ldots, -n_{b_m}$ from $M(-n_1, \ldots, -n_l)$. For convenience, we also define $q_{1, \ldots, l}=|\det(M(\emptyset))|=1$.
\item $u_s := q_{s, \ldots, l} = |\det(M(-n_1, \ldots, -n_{s-1}))|$ $(2 \leq s \leq l)$, $u_0=0, u_1 = 1$.
\item $v_s := q_{1, \ldots, s} = |\det(M(-n_{s+1}, \ldots, -n_l))|$ $(1 \leq s \leq l-1)$, $v_l=1, v_{l+1}=0$.
\item $q = |\det(M(-n_1, \ldots, -n_l))|=u_{l+1}=v_0$.
\item $|[n_1, \ldots, n_l]|:=|\det (M(-n_1, \ldots, -n_l))|$.
\end{enumerate}
It follows that $ q_1 = |\det(M(-n_2,\ldots, -n_l)|=v_1$. The following properties of Hirzebruch-Jung continued fractions will be used.

\begin{lemma}\label{uv} For $1 \leq i \leq l$, we have the following.
\begin{enumerate}
    \item $a_i = 1 - \frac{u_i + v_i}{q}$.
    \item $u_{i+1} = n_i u_{i} - u_{i-1}$, $v_{i-1} = n_i v_i -v_{i+1}$.
    \item $v_i  u_{i+1}- v_{i+1}u_{i} =v_{i-1}  u_{i} - v_i  u_{i-1}= q. $
    \item $|[n_1,\ldots, n_{i-1}, n_i +1, n_{i+1}, \ldots, n_l]|=v_i u_{i}+|[n_1, n_2, \ldots, n_l]|>q.$
    \item If $l \geq 2$ and $n_1 \geq 3$, then $v_1 + v_2 < q$.
\end{enumerate}
\end{lemma}

\begin{proof}
(1)-(4) are well-known facts (for (1), see \cite[Lemma 2.2]{HK11}, and for (2)-(4), see \cite[Lemma 2.4]{HK09}). It is straightforward to check (5) as follows:

$q    = v_1u_2-v_2u_1  =  v_1 \cdot n_1 - v_2 = (n_1-1)v_1 +(v_1-v_2)
      > 2v_1 > v_1 + v_2$.
\end{proof}

Now, we give a criterion for checking $a_k + a_{k+1} \geq 1$ for some $k$. Recall that the intersection point of $E_K$ and $E_{k+1}$ is a redundant point in $S$.

\begin{lemma}\label{TFAE} The following are equivalent.
 \begin{enumerate}
    \item $a_k + a_{k+1} \geq 1$ for some $1 \leq k \leq l-1$.
    \item $q \geq u_k + u_{k+1} + v_k + v_{k+1}$ for some $1 \leq k \leq l-1$.
 \end{enumerate}
\end{lemma}

\begin{proof}
By Lemma \ref{uv} (1), we have
$$a_k + a_{k+1} = \left(1 - \frac{u_k + v_k}{q} \right) + \left(1 - \frac{u_{k+1} + v_{k+1}}{q} \right) = 2 - \frac{ u_k + u_{k+1} + v_k + v_{k+1} }{q},$$
and hence, the equivalence follows.
\end{proof}

The singularity $[\underbrace{2,\ldots,2}_{l}]$ is a canonical singularity of type $A_l$, and hence, $q=l+1$ and $a_i=0$ for all $1 \leq i \leq l$. On the other hand, the minimal resolution $S$ of the singularity $[\underbrace{2,\ldots,2}_{l-1},3]$ does not contain any redundant point. Indeed, $q=2l+1$ and $u_i=i$, $v_i=2l-2i+1$ for $1 \leq i \leq l$ by using Lemma \ref{uv}. Then, for $1 \leq i \leq l-1$,
$$u_i + u_{i+1} + v_i + v_{i+1} = 4l - 2i + 1 > 2l+1=q,$$
and hence, by Lemma \ref{TFAE}, there is no redundant point on $S$.

\begin{proposition}\label{A}
Let $s \in \bar{S}$ be a log terminal singularity of type $A_{q,q_1}$. Then, $S$ does not contain a redundant point if and only if $s$ is of type
$[\underbrace{2,\ldots,2}_{\alpha}] (\alpha \geq 1)$, $[\underbrace{2,\ldots,2}_{\alpha},3] (\alpha \geq 1)$, $[2,2,3,2]$, $[2,3,2]$, $[2,4]$, or $[n]$ for $n \geq 3$.
\end{proposition}

To prove the proposition, we need two more lemmas.

\begin{lemma}\label{A1}
The following hold.
 \begin{enumerate}
    \item If $s$ is the singularity $[\underbrace{2,\ldots,2}_{\alpha},3,\underbrace{2,\ldots,2}_{\beta}]$ $(\alpha \geq \beta \geq 1)$, then $q=\alpha \beta + 2 \alpha + 2 \beta + 3$ and the intersection point of $E_{\alpha}$ and $E_{\alpha+1}$ is a redundant point, except for $[2,2,3,2]$ and $[2,3,2]$.
    \item If $s$ is the singularity  $[\underbrace{2,\ldots,2}_{\alpha},3,\underbrace{2,\ldots,2}_{\beta}, 3, \underbrace{2,\ldots,2}_{\gamma}]$ $(\alpha \geq 0$, $ \beta \geq 1$ and $ \gamma \geq 0)$, then the intersection point of $E_{\alpha+1}$ and $E_{\alpha+2}$ is a redundant point.
    \item If $s$ is the singularity $[\underbrace{2,\ldots,2}_{\alpha},4,\underbrace{2,\ldots,2}_{\beta}]$ $(\alpha \geq \beta \geq 0)$, then the intersection point of $E_{\alpha}$ and $E_{\alpha+1}$ is a redundant point, except for $[2,4]$ and $[4]$.
 \end{enumerate}
\end{lemma}

\begin{proof}
The strategy is as follows. First, we compute $u_{\alpha}, u_{\alpha+1}, u_{\alpha+2}, v_{\alpha}, v_{\alpha+1}, v_{\alpha+2}$ and $q$ using Lemma \ref{uv}. Second, we determine whether
$$q \geq u_{\alpha}+u_{\alpha+1}+v_{\alpha}+v_{\alpha+1}  \quad \text{ or } \quad q \geq u_{\alpha+1}+u_{\alpha+2}+v_{\alpha+1}+v_{\alpha+2}$$
holds or not. Finally, by applying Lemma \ref{TFAE}, we can find a redundant point.

(1) We have
\begin{displaymath}
\begin{array}{l}
u_{\alpha} = |[\underbrace{2,\ldots,2}_{\alpha-1}]| = \alpha, v_{\alpha} = |[3,\underbrace{2,\ldots,2}_{\beta}]| = 2\beta + 3\\
u_{\alpha+1} = |[\underbrace{2,\ldots,2}_{\alpha}]| = \alpha+1, v_{\alpha+1}= |[\underbrace{2,\ldots,2}_{\beta}]| = \beta + 1.\\
\end{array}
\end{displaymath}
Then, $u_{\alpha}+u_{\alpha+1}+v_{\alpha}+v_{\alpha+1}=2 \alpha + 3\beta + 5$.
Using Lemma \ref{uv} (3), we obtain
$$q= v_{\alpha} u_{\alpha+1} - v_{\alpha+1}u_{\alpha} = (\alpha+1)(2\beta+3)-(\beta+1)\alpha=\alpha \beta + 2\alpha + 2\beta + 3.$$
Then, we have
$q -(u_{\alpha}+u_{\alpha+1}+v_{\alpha}+v_{\alpha+1}) = \alpha \beta - \beta - 2 = (\alpha - 1) \beta -2$.
If $\alpha \beta - \beta -2 <0$, then $(\alpha, \beta)=(1,1)$ or $(2,1)$. These cases are exactly $[2,2,3,2]$ and $[2,3,2]$, and we can easily check that there is no redundant point on $S$ (see Table \ref{A.A}). Otherwise, we have $q \geq (u_{\alpha}+u_{\alpha+1}+v_{\alpha}+v_{\alpha+1})$.

(2) We have
\begin{displaymath}
\begin{array}{l}
u_{\alpha+1} =|[\underbrace{2,\ldots,2}_{\alpha}]| = \alpha+1, v_{\alpha+1}=|[\underbrace{2,\ldots,2}_{\beta},3,\underbrace{2,\ldots,2}_{\gamma}]|=\beta\gamma + 2\beta + 2\gamma + 3\\
u_{\alpha+2} =|[\underbrace{2,\ldots,2}_{\alpha},3]| = 2\alpha+3, v_{\alpha+2}=|[\underbrace{2,\ldots,2}_{\beta-1},3,\underbrace{2,\ldots,2}_{\gamma}]|=\beta\gamma + 2\beta + \gamma + 1.
\end{array}
\end{displaymath}
Then, $u_{\alpha+1}+u_{\alpha+2}+v_{\alpha+1}+v_{\alpha+2} = 2\beta \gamma + 3\alpha + 4\beta + 3 \gamma + 8$. Using Lemma \ref{uv} (4), we obtain
\begin{eqnarray*}
q     &=& |[\underbrace{2,\ldots,2}_{\alpha+\beta+1},3,\underbrace{2\ldots,2}_{\gamma}]| + u_{\alpha+1}v_{\alpha+1}\\
      &=& (\alpha+\beta+1)\gamma + 2 (\alpha+\beta+1) + 2 \gamma + 3 + (\alpha+1)(\beta \gamma + 2 \beta + 2 \gamma + 3)\\
      & =& \alpha \beta \gamma + 2\alpha \beta + 3 \alpha \gamma + 2 \beta\gamma + 5\alpha + 4\beta + 5\gamma + 8.
\end{eqnarray*}
It immediately follows that $q \geq u_{\alpha+1}+u_{\alpha+2}+v_{\alpha+1}+v_{\alpha+2}$.

(3) We have
\begin{displaymath}
\begin{array}{l}
u_{\alpha}=|[\underbrace{2,\ldots,2}_{\alpha-1}]|=\alpha, v_{\alpha}=|[4,\underbrace{2,\ldots,2}_{\beta}]|= 3\beta + 4 \\
u_{\alpha+1} = |[\underbrace{2,\ldots,2}_{\alpha}]|=\alpha+1, v_{\alpha+1}= |[\underbrace{2,\ldots,2}_{\beta}]|=\beta+1.
\end{array}
\end{displaymath}
Then, $u_{\alpha}+u_{\alpha+1}+v_{\alpha}+v_{\alpha+1}=2\alpha + 4\beta + 6$.  Using Lemma \ref{uv} (4), we obtain
$$q = |[\underbrace{2,\ldots,2}_{\alpha},3,\underbrace{2,\ldots,2}_{\beta}]| + u_{\alpha+1}v_{\alpha+1}=2\alpha \beta + 3\alpha + 3 \beta + 4.$$
Then, we have $q - (u_{\alpha}+u_{\alpha+1}+v_{\alpha}+v_{\alpha+1}) = 2\alpha \beta+\alpha - \beta - 2$.
If $2 \alpha \beta + \alpha - \beta -2 <0$, then $( \alpha, \beta )=(0,0), (1,0)$. These cases are exactly $[2,4]$ and $[4]$, and we can easily check that there is no redundant point on $S$ (see Table \ref{A.A}). Otherwise, we have $q \geq (u_{\alpha}+u_{\alpha+1}+v_{\alpha}+v_{\alpha+1})$.
\end{proof}

\begin{table}[ht]
\caption{}\label{A.A}
\renewcommand\arraystretch{1.5}
\noindent\[
\begin{array}{|l|l|l|l|l|}
\hline
\text{singularities} &  [2,2,3,2]  & [2,3,2] &  [2,4] & [n] \\
\hline
\text{discrepancies} &\left( \frac{2}{11}, \frac{4}{11}, \frac{6}{11}, \frac{3}{11} \right)  & \left( \frac{1}{4},\frac{1}{2}, \frac{1}{4} \right)  &\left( \frac{2}{7}, \frac{4}{7} \right) & \left( \frac{n-2}{n} \right) (n \geq 2)\\
\hline
\end{array}
\]
\end{table}

\begin{lemma}\label{A2}
Suppose that for $[n_1, \ldots,n_j,\ldots,n_k,\ldots  n_l]$, there are integers $j$ and $k$ with $1 \leq j \leq k \leq l-1$ such that $n_j \geq 3$ and $u_j + u_{j+1} +v_{j}+v_{j+1} \leq q$. Let
$$[n_1', \ldots,n_j',\ldots,n'_{k-1},n_k',n'_{k+1},\ldots ,  n_l']:=[n_1, \ldots,n_j,\ldots,n_{k-1},n_k+1,n_{k+1} , \ldots , n_l].$$
Then, we have $u_j' + u_{j+1}' +v_{j}'+v_{j+1}' \leq q'$.
\end{lemma}

\begin{proof}
By Lemma \ref{uv}, $q'=q+u_kv_k$. It suffices to show that $u_j' + u_{j+1}' +v_{j}'+v_{j+1}' \leq q+u_kv_k$. In order to calculate $u_j', u_{j+1}', v_{j}'$, and $v_{j+1}'$, we divide it into three cases.\\[5pt]
Case 1: $j+2 \leq k$.

We have $u_{j}' = u_j$ and $u_{j+1}'= u_{j+1}$. Moreover, we have
\begin{displaymath}
\begin{array}{l}
v_{j}'=|[n_{j+1},\ldots,n_{k-1},n_{k}+1,n_{k+1},\ldots, n_l]| = v_j + v_k|[n_{j+1},\ldots, n_{k-1}]| \text{ and} \\
v_{j+1}'=|[n_{j+2},\ldots,n_{k-1},n_{k}+1,n_{k+1},\ldots, n_l]| = v_{j+1} + v_k|[n_{j+2},\ldots, n_{k-1}]|
\end{array}
\end{displaymath}
Then, $u_j' + u_{j+1}' +v_{j}'+v_{j+1}' \leq q'$ is equivalent to
\begin{displaymath}
\begin{array}{l}
 u_j + u_{j+1} +v_{j}+v_{j+1} +v_k(|[n_{j+1},\ldots, n_{k-1}]|+|[n_{j+2},\ldots, n_{k-1}]|)\\
 \leq q + u_k v_k = q+v_k |[n_{j},\ldots, n_{k-1}]|. \\
\end{array}
\end{displaymath}
The above inequality always holds by the assumption and Lemma \ref{uv} (5).\\[5pt]
Case 2: $j+1 = k$.

We have $u_{j}' =u_j$ and $u_{j+1}'=u_{j+1}$. Moreover, we have
$$v_{j}'=|[n_{j+1}+1,n_{j+2},\ldots, n_l]| = v_{j} + v_{j+1} \text{ and } v_{j+1}'=|[n_{j+2},\ldots, n_l]| = v_{j+1}.$$
Since $v_{j+1} \leq u_{j+1}v_{j+1}$, we obtain $ u_j' + u_{j+1}' +v_{j}'+v_{j+1}' \leq q + u_{j+1}v_{j+1}$.\\[5pt]
Case 3: $j=k$.

We have $u_{j}'=u_j$ and $u_{j+1}'=|[n_1,\ldots,n_j +1]|=u_{j+1}+u_j$. Moreover, we also have $v_{j}'= v_{j}$ and $v_{j+1}'= v_{j+1}$. As in Case 2, we obtain $ u_j' + u_{j+1}' +v_{j}'+v_{j+1}' \leq q + u_{j+1}v_{j+1}$.
\end{proof}

\begin{proof}[Proof of Proposition \ref{A}]
First, we introduce notations for simplicity. Fix a natural number $l$. For integers $n_1, \ldots, n_l \geq 2$ and $n_1', \ldots, n_l' \geq 2$, we write $[n_1, \ldots, n_l] < [n_1', \ldots, n_l']$ if and only if $n_i \leq n_i'$ for all $1 \leq i \leq l$ and $n_j < n_j'$ for some $1 \leq j \leq l$.

Suppose that $s \in \bar{S}$ is the singularity $[n_1,\ldots,n_l]$ other than
$[\underbrace{2,\ldots,2}_{\alpha}] (\alpha \geq 1)$, $[\underbrace{2,\ldots,2}_{\alpha},3] (\alpha \geq 1)$, $[2,2,3,2]$, $[2,3,2]$, $[2,4]$, or $[n]$ for $n \geq 3$, in which cases $S$ has no redundant point.
Then, $[n_1,\ldots,n_l]$ is greater than or equal to
$[\underbrace{2,\ldots,2}_{\alpha},3, \underbrace{2,\ldots,2}_{\beta}] (\alpha$$ \geq \beta \geq 1)$
(not equal to $[2,2,3,2]$ and $[2,3,2]$),
$[\underbrace{2,\ldots,2}_{\alpha},3, \underbrace{2,\ldots,2}_{\beta}, 3, \underbrace{2,\ldots,2}_{\gamma}]$ $ (\alpha \geq 0, \beta \geq 1 \text{ and } \gamma \geq 0)$,
$[\underbrace{2,\ldots,2}_{\alpha},4, \underbrace{2,\ldots,2}_{\beta}] $ $(\alpha \geq \beta \geq 0)$
(not equal to $[2,4]$ and $[4]$),
$[2,2,3,3]$, $[2,3,3,2]$, $[2,3,3]$, $[3,3]$, or $[2,5]$, in which cases $S$ has a redundant point thanks to Lemma \ref{A1} and Table \ref{A.A.A}. Thus, by Lemma \ref{A2}, $[n_1,\ldots,n_l]$ has a redundant point.
\end{proof}

\begin{table}[ht]
\caption{}\label{A.A.A}
\renewcommand\arraystretch{1.5}
\noindent\[
\begin{array}{|l|l|l|l|l|l|}
\hline
\text{singularities} &  [2,2,3,3]  & [2,3,3,2] & [2,3,3] &  [3,3] & [2,5] \\
\hline
\text{discrepancies}
&\left( \frac{2}{9}, \frac{4}{9}, \frac{6}{9}, \frac{5}{9} \right)
& \left( \frac{1}{3},\frac{2}{3}, \frac{2}{3}, \frac{1}{3} \right)
&\left( \frac{4}{13}, \frac{8}{13}, \frac{7}{13} \right)
& \left( \frac{1}{2}, \frac{1}{2} \right)
& \left( \frac{1}{3}, \frac{2}{3} \right)\\
\hline
\end{array}
\]
\end{table}

\subsection{$D_{q,q_1}$-type}

The dual graph of  a log terminal singularity $s \in \bar{S}$ of $D_{q,q_1}$-type is

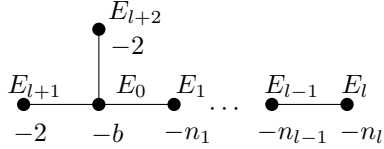
\begin{tikzpicture}[line cap=round,line join=round,>=triangle 45,x=1.0cm,y=1.0cm]
\clip(-4.3,1.1) rectangle (7.28,3.8);
\draw (0.22,2.78)-- (0.22,1.78);
\draw (-0.78,1.78)-- (0.22,1.78);
\draw (0.22,1.78)-- (1.22,1.78);
\draw (2.52,1.78)-- (3.52,1.78);
\draw (0.22,3.28) node[anchor=north west] {$E_{l+2}$};
\draw (-1.12,2.3) node[anchor=north west] {$E_{l+1}$};
\draw (0.32,2.3) node[anchor=north west] {$E_0$};
\draw (1.12,2.3) node[anchor=north west] {$E_1$};
\draw (2.3,2.3) node[anchor=north west] {$E_{l-1}$};
\draw (3.32,2.3) node[anchor=north west] {$E_l$};
\draw (0.25,2.8) node[anchor=north west] {$-2$};
\draw (-1.04,1.65) node[anchor=north west] {$-2$};
\draw (0,1.65) node[anchor=north west] {$-b$};
\draw (0.96,1.65) node[anchor=north west] {$-n_1$};
\draw (2.2,1.65) node[anchor=north west] {$-n_{l-1}$};
\draw (3.3,1.65) node[anchor=north west] {$-n_l$};
\draw (1.56,1.92) node[anchor=north west] {$\cdots$};
\begin{scriptsize}
\fill [color=black] (-0.78,1.78) circle (2.5pt);
\fill [color=black] (0.22,1.78) circle (2.5pt);
\fill [color=black] (0.22,2.78) circle (2.5pt);
\fill [color=black] (1.22,1.78) circle (2.5pt);
\fill [color=black] (2.52,1.78) circle (2.5pt);
\fill [color=black] (3.52,1.78) circle (2.5pt);
\end{scriptsize}
\end{tikzpicture}\\
where $\frac{q}{q_1} = [n_1, \ldots, n_l]$, and $b \geq 2$ and $n_i \leq 2$ are integers for all $1 \leq i \leq l$. The matrix equation for each discrepancy $a_i$ is given by
\begin{displaymath}
\left( \begin{array}{cccccccc}
-2 & 0 & 1 & 0 & 0 &0& \cdots & 0 \\
0 & -2 & 1 & 0 & 0 &0& \cdots & 0 \\
1 & 1 & -b & 1 & 0 &0& \cdots &0 \\
0 & 0 & 1 & -n_1 & 1 &0& \cdots & 0 \\
0 & 0 & 0 & 1 & -n_2 & 1 &\cdots & 0 \\
0 & 0 & 0 & 0 & 1 & -n_3 &\cdots & 0 \\
\vdots & \vdots & \vdots & \vdots & \vdots & \vdots & \ddots & \vdots\\
0 & 0 & 0 & 0 & 0 & \cdots  & 1 & -n_l \\
\end{array} \right)
\left( \begin{array}{c}
a_{l+1}\\
a_{l+2}\\
a_0\\
a_1\\
a_2\\
a_3\\
\vdots\\
a_l\\
\end{array} \right)
=-\left( \begin{array}{c}
0\\
0\\
b-2\\
n_1-2\\
n_2-2\\
n_3-2\\
\vdots\\
n_l-2\\
\end{array} \right).
\end{displaymath}

\begin{lemma}[{\cite[Lemma 3.7]{HK11}}]\label{lemD}
We have the following.
 \begin{enumerate}
    \item $a_0 = 1 - \frac{1}{(b-1)q-q_1}$ and $a_{l+1}=a_{l+2}= \frac{1}{2}a_0$.
    \item $a_1 = 1 - \frac{b-1}{(b-1)q-q_1}$.
    \item $a_l = 1 - \frac{(b-1)q_l - q_{1,l}}{(b-1)q - q_1}$ for $l \geq 2$.
 \end{enumerate}
\end{lemma}

\begin{proposition}\label{D}
Let $s \in \bar{S}$ be a log terminal singularity of $D_{q,q_1}$-type. Then, we have the following four cases.
 \begin{enumerate}
    \item If $b \geq 3$, then $a_0 + a_{l+1} \geq 1$.
    \item If $b = 2$ and $q \geq q_1 + 3$, then $a_0 + a_{l+1} \geq 1$.
    \item If $b = 2$ and $q = q_1 + 2$, then $a_0=a_1=\frac{1}{2}$.
    \item If $b = 2$ and $q = q_1 + 1$, then $s$ is a canonical singularity.
 \end{enumerate}
In particular, the minimal resolution $S$ of $\bar{S}$ always has a redundant point, unless $s$ is a canonical singularity.
\end{proposition}

\begin{proof}
Using Lemma \ref{lemD}, we get
$$a_0 + a_{l+1} = \frac{3}{2}a_0 = \frac{3}{2}\left(1- \frac{1}{(b-1)q-q_1} \right) \geq 1$$
which is equivalent to
$$(b-1)q-q_1 \geq 3.$$
If $b \geq 3$, then
$$(b-1)q-q_1 \geq 2q-q_1 > q \geq 2.$$
Thus, $(b-1)q-q_1 \geq 3$, that is, $a_0 + a_{l+1} \geq 1$, which proves (1).

Suppose that $b=2$. If $q-q_1 \geq 3$, then it is still true that $a_0 + a_{l+1} \geq 1$. This proves (2).

Now, we assume that $b=2$ and $q=q_1+2$. The condition $q=q_1+2$ is equivalent to $(n_1 - 1)q_1 = q_{1,2} + 2$. By simple calculation, we conclude that $n_1=2$ and $q_1 = q_{1,2} + 2$. By induction,
$$[n_1, \ldots, n_{l-1}, n_l] = [2, \ldots, 2, 3].$$
Using Lemma \ref{lemD}, we have $a_0=a_1=\frac{1}{2}$, which proves (3).

Finally, we assume that $b=2$ and $q=q_1 + 1$. As in the proof of (3), we obtain
$$[n_1, \ldots, n_l] = [2, \ldots, 2, 2].$$
Thus, $y$ is a canonical singularity of type $D_{l+3}$, which proves (4).
\end{proof}

\subsection{$T_m$, $O_m$, and $I_m$-types}

We use the notation $\langle b;q,q_1 ; q',q_1' \rangle$ to refer to the following dual graph

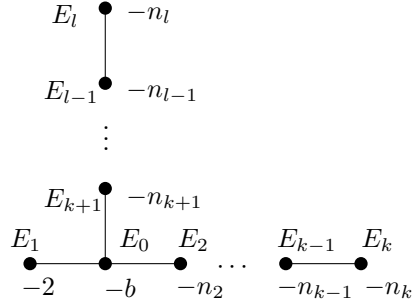
\begin{tikzpicture}[line cap=round,line join=round,>=triangle 45,x=1.0cm,y=1.0cm]
\clip(-4.3,0.0) rectangle (7.28,4.7);
\draw (-0.16,4.02)-- (-0.16,3.02);
\draw (-0.16,1.62)-- (-0.16,0.62);
\draw (-1.16,0.62)-- (-0.16,0.62);
\draw (-0.16,0.62)-- (0.84,0.62);
\draw (2.24,0.62)-- (3.24,0.62);
\draw (-0.32,2.78) node[anchor=north west] {$\vdots$};
\draw (1.2,0.78) node[anchor=north west] {$\cdots$};
\draw (-1,4.18) node[anchor=north west] {$E_l$};
\draw (-1.1,3.18) node[anchor=north west] {$E_{l-1}$};
\draw (-1.1,1.74) node[anchor=north west] {$E_{k+1}$};
\draw (-1.55,1.2) node[anchor=north west] {$E_1$};
\draw (-1.4,0.55) node[anchor=north west] {$-2$};
\draw (-0.1,1.2) node[anchor=north west] {$E_0$};
\draw (0.68,1.2) node[anchor=north west] {$E_2$};
\draw (2,1.2) node[anchor=north west] {$E_{k-1}$};
\draw (3.12,1.2) node[anchor=north west] {$E_k$};
\draw (-0.3,0.55) node[anchor=north west] {$-b$};
\draw (0.66,0.55) node[anchor=north west] {$-n_2$};
\draw (2,0.55) node[anchor=north west] {$-n_{k-1}$};
\draw (3.16,0.55) node[anchor=north west] {$-n_k$};
\draw (0.0,4.2) node[anchor=north west] {$-n_l$};
\draw (0.0,3.16) node[anchor=north west] {$-n_{l-1}$};
\draw (0.0,1.74) node[anchor=north west] {$-n_{k+1}$};
\begin{scriptsize}
\fill [color=black] (-1.16,0.62) circle (2.5pt);
\fill [color=black] (-0.16,0.62) circle (2.5pt);
\fill [color=black] (0.84,0.62) circle (2.5pt);
\fill [color=black] (-0.16,1.62) circle (2.5pt);
\fill [color=black] (-0.16,3.02) circle (2.5pt);
\fill [color=black] (-0.16,4.02) circle (2.5pt);
\fill [color=black] (2.24,0.62) circle (2.5pt);
\fill [color=black] (3.24,0.62) circle (2.5pt);
\end{scriptsize}
\end{tikzpicture}\\
where $b \geq 2$ and $n_i \geq 2$ are integers for all  $2 \leq i \leq l$, and $\frac{q}{q_1}=[n_2, \ldots, n_k]$ and $\frac{q'}{q_1'}=[n_{k+1},\ldots,n_l]$ are the Hirzebruch-Jung continued fractions. Each dual graph of a log terminal singularity $s \in \bar{S}$ of $T_m$, $I_m$, or $O_m$-types is one of three items from the top, eight items from the bottom, or the remaining items in Table \ref{T.T}, respectively.

For example, consider $\langle b;3,1;4,3\rangle$, which is of $O_m$-type.

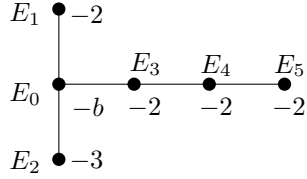
\begin{tikzpicture}[line cap=round,line join=round,>=triangle 45,x=1.0cm,y=1.0cm]
\clip(-4.34,0.4) rectangle (7.24,3.8);
\draw (0.18,2.96)-- (0.18,1.96);
\draw (0.18,0.96)-- (0.18,1.96);
\draw (1.18,1.96)-- (0.18,1.96);
\draw (1.18,1.96)-- (2.18,1.96);
\draw (2.18,1.96)-- (3.18,1.96);
\draw (-0.6,3.16) node[anchor=north west] {$E_1$};
\draw (-0.6,2.1) node[anchor=north west] {$E_0$};
\draw (-0.6,1.2) node[anchor=north west] {$E_2$};
\draw (0.2,3.12) node[anchor=north west] {$-2$};
\draw (0.23,1.9) node[anchor=north west] {$-b$};
\draw (0.2,1.2) node[anchor=north west] {$-3$};
\draw (0.97,1.9) node[anchor=north west] {$-2$};
\draw (1.92,1.9) node[anchor=north west] {$-2$};
\draw (2.89,1.9) node[anchor=north west] {$-2$};
\draw (1,2.5) node[anchor=north west] {$E_3$};
\draw (1.94,2.48) node[anchor=north west] {$E_4$};
\draw (2.92,2.48) node[anchor=north west] {$E_5$};
\begin{scriptsize}
\fill [color=black] (0.18,2.96) circle (2.5pt);
\fill [color=black] (0.18,1.96) circle (2.5pt);
\fill [color=black] (0.18,0.96) circle (2.5pt);
\fill [color=black] (1.18,1.96) circle (2.5pt);
\fill [color=black] (2.18,1.96) circle (2.5pt);
\fill [color=black] (3.18,1.96) circle (2.5pt);
\end{scriptsize}
\end{tikzpicture}\\
The matrix equation for discrepancies $a_i$ is given by
\begin{displaymath}
\left( \begin{array}{cccccc}
-b & 1 & 1 & 1 & 0 & 0 \\
1  &-2 & 0 & 0 & 0 & 0 \\
1  & 0 &-3 & 0 & 0 & 0\\
1  & 0 & 0 &-2 & 1 & 0 \\
0  & 0 & 0 & 1 &-2 & 1 \\
0  & 0 & 0 & 0 & 1 &-2
\end{array} \right)
\left( \begin{array}{c} a_0 \\ a_1 \\a_2 \\ a_3 \\a_4 \\a_5 \end{array} \right) =
\left( \begin{array}{c} 2-b \\ 0 \\ -1 \\ 0 \\ 0 \\0 \end{array} \right).
\end{displaymath}
It is easy to see that the solution $(a_0, a_2, a_2, a_3, a_4, a_5)$ is
$$\left( \frac{12b-20}{12b-19} , \frac{6b-10}{12b-19} , \frac{8b-13}{12b-19} ,  \frac{9b-15}{12b-19} ,  \frac{6b-12}{12b-19} ,  \frac{3b-5}{12b-19}  \right).  $$
Since $b \geq 2$, we obtain $a_0 + a_1 \geq 1$ by the following computation:
$$ \frac{12b-20}{12b-19} + \frac{6b-10}{12b-19} =  \frac{18b-30}{12b-19} \geq 1$$
is equivalent to
$$ 6b \geq 11.$$
Similarly, we calculate all discrepancies for 15 cases. For the reader's convenience, we list discrepancies $(a_0, a_1, \ldots, a_l)$ of all possible cases in Table \ref{T.T}. It is easy to check that $a_0 + a_1 \geq 1$ for all cases provided that $b \geq 2$ and $y$ is not a canonical singularity. Thus, we obtain the following.

\begin{proposition}\label{TOI}
If $s \in \bar{S}$ is a log terminal singularity of one of $T_m$, $O_m$, and $I_m$-types, and $s$ is not a canonical singularity, then $a_0 + a_1 \geq 1$.
\end{proposition}

\begin{table}[ht]
\caption{}\label{T.T}
\renewcommand\arraystretch{1.5}
\noindent\[
\begin{array}{|l|l|}
\hline
\hbox{singularities} & \hbox{discrepancies } (a_0, a_1, \ldots, a_l)\\
\hline
\langle b;3,1;3,1\rangle &  \left( \frac{6b-8}{6b-7} , \frac{3b-4}{6b-7} , \frac{4b-5}{6b-7}  ,  \frac{4b-5}{6b-7} \right)\\
\hline
\langle b;3,1;3,2\rangle & \left( \frac{6b-10}{6b-9} , \frac{3b-5}{6b-9}, \frac{12b-19}{18b-27}  ,  \frac{12b-20}{18b-27} ,  \frac{6b-10)}{18b-27)}\right) \\
\hline
\langle b;3,2;3,2\rangle & \left( \frac{6b-12}{6b-11} , \frac{3b-6}{6b-11}, \frac{4b-8}{6b-11}  ,  \frac{2b-4}{6b-11} , \frac{4b-8}{6b-11} ,  \frac{2b-4}{6b-11}  \right)\\
\hline
\langle b;3,2;4,3\rangle & \left( \frac{12b-24}{12b-23} , \frac{6b-12}{12b-23} , \frac{8b-16}{12b-23}  ,  \frac{4b-8}{12b-23},  \frac{9b-18}{12b-23},  \frac{6b-12}{12b-23},  \frac{3b-6}{12b-23} \right) \\
\hline
\langle b;3,1;4,3\rangle & \left( \frac{12b-20}{12b-19} , \frac{6b-10}{12b-19} , \frac{8b-13}{12b-19} , \frac{9b-15}{12b-19} ,  \frac{6b-10}{12b-19},  \frac{3b-5}{12b-19}  \right) \\
\hline
\langle b;3,2;4,1\rangle & \left( \frac{12b-18}{12b-17} , \frac{6b-9}{12b-17} , \frac{8b-12}{12b-17}  ,  \frac{4b-6}{12b-17} ,  \frac{9b-13}{12b-17} \right) \\
\hline
\langle b;3,1;4,1\rangle & \left( \frac{12b-14}{12b-13} , \frac{6b-7}{12b-13} ,   \frac{8b-9}{12b-13} , \frac{9b-10}{12b-13} \right) \\
\hline
\langle b;3,2;5,4\rangle & \left( \frac{30b-60}{30b-59} , \frac{15b-30}{30b-59} , \frac{20b-40}{30b-59}  ,  \frac{10b-20}{30b-59} , \frac{24b-48}{30b-59}, \frac{18b-36}{30b-59} ,  \frac{12b-24}{30b-59} ,  \frac{6b-12}{30b-59} \right) \\
\hline
\langle b;3,2;5,3\rangle & \left(\frac{30b-54}{30b-53} ,\frac{15b-27}{30b-53} , \frac{20b-36}{30b-53}  ,  \frac{10b-18}{30b-53} ,  \frac{24b-43}{30b-53} ,  \frac{18b-32}{30b-53} \right) \\
\hline
\langle b;3,1;5,4\rangle & \left( \frac{30b-50}{30b-49} , \frac{15b-25}{30b-49} , \frac{20b-33}{30b-49}  ,  \frac{24b-40}{30b-49} , \frac{18b-30}{30b-49} ,  \frac{12b-20}{30b-49} , \frac{6b-10}{30b-49}  \right) \\
\hline
\langle b;3,2;5,2\rangle & \left( \frac{30b-48}{30b-47} , \frac{15b-24}{30b-47} , \frac{20b-32}{30b-47}  ,  \frac{10b-16}{30b-47} , \frac{24b-38}{30b-47} , \frac{12b-19}{30b-47}  \right) \\
\hline
\langle b;3,1;5,3\rangle & \left(\frac{30b-44}{30b-43} , \frac{15b-22}{30b-43} , \frac{20b-29}{30b-43}  ,  \frac{24b-35}{30b-43} ,  \frac{18b-26}{30b-43} \right) \\
\hline
\langle b;3,2;5,1\rangle & \left(  \frac{30b-42}{30b-41} , \frac{15b-21}{30b-41} , \frac{20b-28}{30b-41}  ,  \frac{10b-14}{30b-41} ,  \frac{24b-33}{30b-41}  \right) \\
\hline
\langle b;3,1;5,2\rangle & \left(  \frac{30b-38}{30b-37} , \frac{15b-19}{30b-37} , \frac{20b-25}{30b-37}  ,  \frac{24b-30}{30b-37} ,  \frac{12b-15}{30b-37}  \right) \\
\hline
\langle b;3,1;5,1\rangle & \left( \frac{30b-32}{30b-31} , \frac{15b-16}{30b-31} , \frac{20b-21}{30b-31} , \frac{24b-25}{30b-31}  \right) \\
\hline

\end{array}
\]
\end{table}

We are ready to prove Theorem \ref{reddisc}.

\begin{proof}[Proof of Theorem \ref{reddisc}]
It is a consequence of Lemmas \ref{rtsing}, \ref{can}, Propositions \ref{A}, \ref{D}, and \ref{TOI}.
\end{proof}

\section{Examples of redundant blow-ups}\label{exredsec}

In this section, we construct various examples of redundant blow-ups. In particular, we prove Theorem \ref{nomin}, i.e., we answer Question \ref{tvvq}.

\subsection{Geometry of big anticanonical surfaces}\label{prelimbigsursubsec}

In this subsection, we briefly review some basic properties of big anticanonical surfaces.
A smooth projective surface $S$ is called a \emph{big anticanonical surface} if the anticanonical divisor $-K_S$ is big. Although not every anticanonical ring $R(-K_X):=\bigoplus_{m \geq 0} H^0 (\mathcal{O}_X(-mK_S)$ of a big anticanonical surface $S$ is finitely generated (\cite[Lemma 14.39]{B01}), we can always define the anticanonical model of $S$ using the Zariski decomposition $-K_S = P+N$. The morphism $f \colon S \rightarrow \bar{S}$ which contracts all curves $C$ such that $C.P=0$ is called the \emph{anticanonical morphism}, and $\bar{S}$ is called the \emph{anticanonical model} of $S$. Then, $\bar{S}$ is a normal projective surface (\cite[14.32]{B01}).

Note that the anticanonical ring $R(-K_S)$ of a big anticanonical \emph{rational} surface $S$ is always finitely generated, and the anticanonical model $\bar{S}=\Proj R(-K_S)$ is a del Pezzo surface, i.e., $-K_{\bar{S}}$ is an ample $\Q$-Cartier divisor. In this case, $\bar{S}$ contains rational singularities by \cite[Theorem 4.3]{Sak84}. Furthermore, we have the following.

\begin{proposition}[{\cite[Proposition 4.2]{Sak84}}]\label{redbigratsurf}
Let $S$ be a big anticanonical rational surface. Then, there is a sequence of redundant blow-ups $f \colon S \rightarrow S_0$, where $S_0$ is the minimal resolution of the anticanonical model of $S$.
\end{proposition}

\subsection{Redundant blow-ups of big anticanonical rational surfaces}\label{exredsubsec}

In this subsection, we give an explicit construction of big anticanonical rational surfaces containing redundant points.

Before giving constructions, we explain how these examples give a negative answer to Question \ref{tvvq}. Let $S$ be a big anticanonical rational surface admiting a redundant blow-up $f \colon \widetilde{S} \to S$. By Lemma \ref{redlem}, the anticanonical models $\bar{S}$ of $\widetilde{S}$ and $S$ are the same. Suppose that $\widetilde{S}$ is a minimal resolution of a del Pezzo surface. Then, this minimal resolution is nothing but the anticanonical morphism to $\bar{S}$. However, in view of Proposition \ref{redbigratsurf}, $\widetilde{S}$ cannot be the minimal resolution of $\bar{S}$ because the anticanonical morphism $\widetilde{S} \to \bar{S}$ factors through $S$. Thus, the existence of redundant points on some big anticanonical rational surface answers to Question \ref{tvvq}.

\begin{remark}
Let $S$ be a big anticanonical rational surface, and let $f \colon \widetilde{S} \rightarrow S$ be a blow-up at $p \in S$. Even though  $f$ is not a redundant blow-up,  $\widetilde{S}$ might be a big anticanonical rational surface. However, in this case, the anticanonical model of $\widetilde{S}$ is different from that of  $S$.
\end{remark}

In the following examples, we construct big anticanonical rational surfaces whose anticanonical models contain only log terminal singularities (Example \ref{ltex}) or contain at least one rational singularity that is not a log terminal singularity (Example \ref{ctex2}).

\begin{example}\label{ltex}
Let $\pi \colon  S(m,n) \rightarrow \P^2$ be the blow-up of $\P^2$ at $p_1^1, \ldots p_m^1$ on a line $\overline{l}_1$ and $p_1^2, \ldots p_n^2$ on a different line $\overline{l}_2$ for $m \geq n \geq 4$. Assume that all the chosen points are distinct and away from the intersection point of $\overline{l}_1$ and $\overline{l}_2$. Let $l_i$ be the strict transform of $\overline{l}_i$, $E_j^i$ be the exceptional divisors of $p_j^i$, and $l$ be the pull-back of a general line in $\P^2$. Then, the anticanonical divisor of $S(m,n)$ is given by
\[-K_{S(m,n)} = \pi^{*}(-K_{\P^2}) - \sum_{i,j}E_j^i = l+l_1 + l_2.\]
Let $-K_{S(m,n)} = P+N$ be the Zariski decomposition. Then, we have
$$N=\frac{mn-m-2n}{mn-m-n}l_1 + \frac{mn-2m-n}{mn-m-n}l_2.$$
It is easy to see that $S(m,n)$ is a big anticanonical rational surface. By contracting curves $l_1$ and $l_2$ on $S(m,n)$, we obtain the anticanonical model of $S(m,n)$ which contains only log terminal singularities. Now, the intersection point $p$ of $l_1$ and $l_2$ is a redundant point because
$$ \mult_p N=\frac{mn-m-2n}{mn-m-n} + \frac{mn-2m-n}{mn-m-n} \geq 1 .$$
Thus, by blowing up at $p$, we obtain a redundant blow-up $f \colon  \widetilde{S}(m,n) \rightarrow S(m,n)$. The Picard number of $\widetilde{S}(m,n)$ is $m+n+2$, and hence, it supports Theorem \ref{nomin}.
\end{example}

Singularities on log del Pezzo surfaces are classified in some cases (see e.g., \cite{AN}, \cite{Na}, \cite{K}). Thus, we can find redundant points on the minimal resolution of log del Pezzo surfaces.

\begin{example}\label{ctex2}
Let $h \colon \mathbb{F}_n \rightarrow \mathbb{P}^1$ be the Hirzebruch surface with a section $\overline{\sigma}$ of self-intersection $-n \leq -2$.  Let $k$ be an integer such that $3 \leq k \leq n+1$ and let $a_1 , \ldots, a_k$ be positive integers such that $\sum_{j=1}^{k} \frac{1}{a_j} < k-2$. Choose $k$ distinct fibers $\overline{F_1}, \ldots, \overline{F_k}$ of $h$, and choose $a_i$ distinct points $p_1^{i}, \ldots, p_{a_i}^i$ on $\overline{F_i} \backslash \overline{\sigma}$. Let $S$ be the blow-up of $\mathbb{F}_n$ at $p_j^i$ where $1 \leq i \leq k, 1 \leq j \leq a_i$. Let $\sigma$ be the strict transform of $\overline{\sigma}$, and let $F_i$ be the strict transform of $\overline{F_i}$. Denote the pull-back of a general fiber of $\mathbb{F}_n$ by $F$.

Then, we have the Zariski decomposition of the anticanonical divisor $-K_{S}=P+N$ as follows:
\[P=\frac{n+2-k}{n-\sum{\frac{1}{a_j}}} \sigma + (n+2-k)F + \sum_{i=1}^{k}{\frac{n+2-k}{a_i(n-\sum{\frac{1}{a_j}})} F_i} ,\]
\[N= \left( 2-\frac{n+2-k}{n-\sum\frac{1}{a_j}} \right)\sigma + \sum_{i=1}^k \left(1- \frac{n+2-k}{a_i (n-\sum \frac{1}{a_j})} \right)F_i.\]
Thus, $S$ is a big anticanonical rational surface, and every point $p$ in $\sigma$ is a redundant point because
$$\mult_{p}N \geq  2-\frac{n+2-k}{n-\sum\frac{1}{a_j}}  >1.$$
The construction of $S$ appeared in Section 3 of \cite{TVV10}.
\end{example}

\subsection{Redundant blow-ups of other surfaces}

Throughout this subsection, we simply assume that $k=\C$. Here, we construct smooth projective surfaces with $\kappa(-K)=2,1,$ or $0$ containing redundant points.

\begin{example}
Let $C$ be a smooth projective curve of genus $g \geq 1$, and let $A$ be a divisor of degree $e > 2g-2$ on $C$. Consider the ruled surface $X:=\P(\mathcal{O}_C \oplus \mathcal{O}_C(-A))$ with the ruling $\pi \colon X \to C$. Then, $-K_X$ is big. It is easy to see that the negative part of the Zariski decomposition $-K_X = P+N$ is given by $N=\left(1+\frac{2g-2}{e} \right)C_0$,
where $C_0$ is a section of $\pi$ corresponding to the canonical projection $\mathcal{O}_C(-A) \oplus \mathcal{O}_C \to \mathcal{O}_C(-A)$.  Thus, every point on $C_0$ is redundant.
\end{example}

\begin{example}\label{x22}
Let $S$ be an extremal rational elliptic surface $X_{321}$ in \cite[Theorem 4.1]{MP86}. There is a singular fiber which consists of two smooth rational curves $A$ and $B$ meeting transversally at two points $p$ and $q$.  Let $\pi \colon \widetilde{S} \rightarrow S$ be the blow-up $p$ with the exceptional divisor $E$. Since $\mult_p(A+B)=2$. By \cite[Lemma 4.4]{AL11}, we obtain the Zariski decomposition $-K_{\widetilde{S}} =P+N$, where $P=\frac{1}{2}\pi^{*}(-K_S)$ and $N = \frac{1}{2}(\pi^{-1}_* A+ \pi^{-1}_* B)$. Then, $\kappa(-K_{\widetilde{S}})=1$, and the intersection point $q$ of $\pi^{-1}_* A$ and $\pi^{-1}_* B$ is redundant.
\end{example}

\begin{example}\label{0redex}
Let $S$ be an extremal rational elliptic surface $X_{22}$ in \cite[Theorem 4.1]{MP86}. The elliptic fibration has a singular fiber $C$ which is a cuspidal rational curve and the unique section $D$. Let $p$ be the intersection point of a singular fiber $C$ and the section $D$, and let $\pi \colon \widetilde{S} \to S$ be the blow-up at $p$ with the exceptional divisor $E$. Then, we obtain the Zariski decomposition $-K_{\widetilde{S}} = P + N$, where $P=0$ and $N = \pi^{-1}_*C$. Thus, $\kappa(-K_{\widetilde{S}})=0$, and every point on $N$ is redundant.
\end{example}


\end{document}